\documentclass[
11pt,  reqno]{amsart}
\usepackage[margin=1.2in,marginparwidth=1.5cm, marginparsep=0.5cm]{geometry}

\usepackage{amsmath,amssymb,amscd,amsthm,amsxtra, esint}

\usepackage{booktabs} 
\usepackage{microtype}
\usepackage{amssymb}
\usepackage{mathrsfs}

\usepackage{upgreek}

\usepackage{color}
\usepackage[implicit=true]{hyperref}

\usepackage{xsavebox}

\usepackage{cases}

\allowdisplaybreaks[2]

\definecolor{gr}{rgb}   {0.,   0.69,   0.23 }
\definecolor{bl}{rgb}   {0.,   0.5,   1. }
\definecolor{mg}{rgb}   {0.85,  0.,    0.85}
\definecolor{yl}{rgb}   {0.8,  0.7,   0.}
\definecolor{or}{rgb}  {0.7,0.2,0.2}

\usepackage{tikz}
\usepackage{marginnote}
\usepackage{scalerel} 

\usetikzlibrary{shapes.misc}
\usetikzlibrary{shapes.symbols}
\usetikzlibrary{decorations}
\usetikzlibrary{decorations.markings}

\tikzset{
	dot/.style={circle,fill=black,draw=black,inner sep=0pt,minimum size=0.5mm},
	>=stealth,
	}

\tikzset{
	dot2/.style={circle,fill=black,draw=black,inner sep=0pt,minimum size=0.2mm},
	>=stealth,
	}

\tikzset{
	ddot/.style={circle,fill=black,draw=black,inner sep=0pt,minimum size=0.8mm},
	>=stealth,
	}

\tikzset{decision/.style={ 
        draw,
        diamond,
        aspect=1.5
    }}

\tikzset{dia2/.style
={diamond,fill=white,draw=black,inner sep=0pt,minimum size=1mm},
	>=stealth,
	}

\tikzset{dia/.style
={star,fill=black,draw=black,inner sep=0pt,minimum size=1mm},
	>=stealth,
	}

\tikzset{dia/.style
={diamond,fill=black,draw=black,inner sep=0pt,minimum size=1.3mm},
	>=stealth,
	}

\makeatletter
\def\DeclareSymbol#1#2#3{\xsavebox{#1}{\tikz[baseline=#2,scale=0.15]{#3}}}
\def\<#1>{\xusebox{#1}}
\makeatother

\newsavebox{\peA}
\newsavebox{\pneA}
\newsavebox{\plA}
\newsavebox{\pgA}
\newsavebox{\pleA}
\newsavebox{\pgeA}
\newsavebox{\pezA}

\savebox{\peA}{\tikz \draw (0,0) node[shape=circle,draw,inner sep=0pt,minimum size=8.5pt] {\scriptsize  $=$};}
\savebox{\pneA}{\tikz \draw (0,0) node[shape=circle,draw,inner sep=0pt,minimum size=8.5pt] {\footnotesize $\neq$};}
\savebox{\plA}{\tikz \draw (0,0) node[shape=circle,draw,inner sep=0pt,minimum size=8.5pt] {\scriptsize $<$};}
\savebox{\pgA}{\tikz \draw (0,0) node[shape=circle,draw,inner sep=0pt,minimum size=8.5pt] {\scriptsize $>$};}
\savebox{\pleA}{\tikz \draw (0,0) node[shape=circle,draw,inner sep=0pt,minimum size=8.5pt] {\scriptsize $\leqslant$};}
\savebox{\pgeA}{\tikz \draw (0,0) node[shape=circle,draw,inner sep=0pt,minimum size=8.5pt] {\scriptsize $\geqslant$};}
\savebox{\pezA}{\tikz \draw (0,0) node[shape=circle,draw,
fill=white, 
inner sep=0pt,minimum size=8.5pt]{} ;}

\def \peB{\mathchoice
{\scalebox{.7}{{\usebox{\peA}}}}
{\scalebox{.7}{{\usebox{\peA}}}}
{\scalebox{.7}{{\usebox{\peA}}}}
{}
}

\def \pezB{\mathchoice
{\scalebox{.7}{{\usebox{\pezA}}}}
{\scalebox{.7}{{\usebox{\pezA}}}}
{\scalebox{.7}{{\usebox{\pezA}}}}
{}
}

\newcommand{\pe}{\mathbin{{\peB}}}

\newcommand{\pez}{\mathbin{{\pezB}}}

\usetikzlibrary{decorations.pathmorphing, patterns,shapes}

\DeclareSymbol{dot}{-2.7}{\draw (0,0) node[dot]{};}

\DeclareSymbol{1}{0}{\draw[white] (-.4,0) -- (.4,0); \draw (0,0)  -- (0,1.2) node[dot] {};}
\DeclareSymbol{2}{0}{\draw (-0.5,1.2) node[dot] {} -- (0,0) -- (0.5,1.2) node[dot] {};}
\DeclareSymbol{3}{0}{\draw (0,0) -- (0,1.2) node[dot] {}; \draw (-.7,1) node[dot] {} -- (0,0) -- (.7,1) node[dot] {};}
\DeclareSymbol{30}{-3}{\draw (0,0) -- (0,-1); \draw (0,0) -- (0,1.2) node[dot] {}; \draw (-.7,1) node[dot] {} -- (0,0) -- (.7,1) node[dot] {};}
\DeclareSymbol{31}{-3}{\draw (0,0) -- (0,-1) -- (1,0) node[dot] {}; \draw (0,0) -- (0,1.2) node[dot] {}; \draw (-.7,1) node[dot] {} -- (0,0) -- (.7,1) node[dot] {};}
\DeclareSymbol{32}{-3}{\draw (0,0) -- (0,-1) -- (1,0) node[dot] {}; \draw (0,0) -- (0,-1) -- (-1,0) node[dot] {}; \draw (0,0) -- (0,1.2) node[dot] {}; \draw (-.7,1) node[dot] {} -- (0,0) -- (.7,1) node[dot] {};}
\DeclareSymbol{20}{-3}{\draw (0,0) -- (0,-1);\draw (-.7,1) node[dot] {} -- (0,0) -- (.7,1) node[dot] {};}
\DeclareSymbol{22}{-3}{\draw (0,0.3) -- (0,-1) -- (1,0) node[dot] {}; \draw (0,0.3) -- (0,-1) -- (-1,0) node[dot] {};\draw (-.7,1) node[dot] {} -- (0,0.3) -- (.7,1) node[dot] {};}
\DeclareSymbol{31p}{-3}{\draw (0,0) -- (0,-1) -- (1,0) node[dot] {}; \draw (0,0) -- (0,1.2) node[dot] {}; \draw (-.7,1) node[dot] {} -- (0,0) -- (.7,1) node[dot] {};
\draw (0,-1) node{\scaleobj{0.5}{\pez}};
\draw (0,-1) node{\scaleobj{0.5}{\pe}}; }
\DeclareSymbol{32p}{-3}{\draw (0,0) -- (0,-1) -- (1,0) node[dot] {}; \draw (0,0) -- (0,-1) -- (-1,0) node[dot] {}; \draw (0,0) -- (0,1.2) node[dot] {}; \draw (-.7,1) node[dot] {} -- (0,0) -- (.7,1) node[dot] {};
\draw (0,-1) node{\scaleobj{0.5}{\pez}};
\draw (0,-1) node{\scaleobj{0.5}{\pe}};}
\DeclareSymbol{22p}{-3}{\draw (0,0.3) -- (0,-1) -- (1,0) node[dot] {}; \draw (0,0.3) -- (0,-1) -- (-1,0) node[dot] {};\draw (-.7,1) node[dot] {} -- (0,0.3) -- (.7,1) node[dot] {};
\draw (0,-1) node{\scaleobj{0.5}{\pez}};
\draw (0,-1) node{\scaleobj{0.5}{\pe}};}
\DeclareSymbol{21p}{-3}{\draw (0,0.3) -- (0,-1) -- (1,0) node[dot] {};
\draw (-.7,1) node[dot] {} -- (0,0.3) -- (.7,1) node[dot] {};
\draw (0,-1) node{\scaleobj{0.5}{\pez}};
\draw (0,-1) node{\scaleobj{0.5}{\pe}};}

\DeclareSymbol{21}{-3}{\draw (0,0.3) -- (0,-1) -- (1,0) node[dot] {};
\draw (-.7,1) node[dot] {} -- (0,0.3) -- (.7,1) node[dot] {};
}

\DeclareSymbol{11p}{-3}{\draw (0,0) node[dot2] {} -- (0,-1) -- (1,0) node[dot] {};
\draw (0,0) -- (0,1.2) node[dot] {};  
\draw (0,-1) node{\scaleobj{0.5}{\pez}};
\draw (0,-1) node{\scaleobj{0.5}{\pe}}; }
\DeclareSymbol{100}{-3}
{\draw[white] (-.4,0) -- (.4,0);
\draw (0,0) node[dot2] {} -- (0,-1)  ;
\draw (0,0) -- (0,1.2) node[dot] {};}  
\def\R{\mathbb{R}}

\def\r2n{{\mathbb{R}^{2n}}}

\def\N{\mathbb{N}}

\def\supp{\operatorname{supp}}

\usetikzlibrary{mindmap,backgrounds,trees,arrows,decorations.pathmorphing,decorations.pathreplacing,positioning,shapes.geometric,shapes.misc,decorations.markings,decorations.fractals,calc,patterns}

\tikzset{>=stealth',
         cvertex/.style={circle,draw=black,inner sep=1pt,outer sep=3pt},
         vertex/.style={circle,fill=black,inner sep=1pt,outer sep=3pt},
         star/.style={circle,fill=yellow,inner sep=0.75pt,outer sep=0.75pt},
         tvertex/.style={inner sep=1pt,font=\scriptsize},
         gap/.style={inner sep=0.5pt,fill=white}}
\tikzstyle{mybox} = [draw=black, fill=blue!10, very thick,
    rectangle, rounded corners, inner sep=10pt, inner ysep=20pt]
\tikzstyle{boxtitle} =[fill=blue!50, text=white,rectangle,rounded corners]

\tikzstyle{decision} = [diamond, draw, fill=blue!20,
    text width=4.5em, text badly centered, node distance=2.5cm, inner sep=0pt]
\tikzstyle{block} = [rectangle, draw, fill=blue!20,
    text width=2.7cm, text centered, rounded corners, minimum height=3em]

\tikzstyle{line} = [draw, very thick, color=black!50, -latex']

\tikzstyle{cloud} = [draw, ellipse,fill=red!40,
node distance=2.5cm,
    minimum height=1em]

\tikzstyle{cloud2} = [draw, ellipse,fill=red!30, text=white,text width=10em, node distance=2.5cm, text centered, minimum height=4em]

\tikzstyle{cloud3} = [draw, ellipse, fill=cyan!30,
node distance=2.5cm,
    minimum height=3em,
    text width=1.7cm]

\tikzstyle{cloud4} = [draw, ellipse,fill=orange!70, node distance=2.5cm,
   text width=2.1cm,
           minimum height=3em]

\tikzstyle{cloud5} = [draw, ellipse,fill=red!20, node distance=2.5cm,
    text centered,
    text width=3.0cm,
    minimum height=3em]

\tikzstyle{cloud6} = [draw, ellipse,fill=red!20, node distance=2.5cm,
    text width=1.5cm,
    text centered,
    minimum height=3em]

\tikzset{
    position/.style args={#1:#2 from #3}{
        at=(#3.#1), anchor=#1+180, shift=(#1:#2)
    }
}

\newtheorem{theorem}{Theorem} [section]

\newtheorem{remark}[theorem]{Remark}

\newtheorem{definition}[theorem]{Definition}

\newtheorem{oldtheorem}{Theorem}

\newcommand{\noi}{\noindent}

\newcommand{\al}{\alpha}
\newcommand{\be}{\beta}

\newcommand{\nb}{\nabla}

\newcommand{\s}{\sigma}

\newcommand{\ft}{\widehat}

\newcommand{\wt}{\widetilde}

\newcommand{\dx}{\partial_x}

\newcommand{\dd}{\partial}

\newcommand{\les}{\lesssim}
\newcommand{\ges}{\gtrsim}

\newcommand{\ind}{\mathbf 1}

\renewcommand{\S}{\mathcal{S}}

\newtheorem*{ackno}{Acknowledgements}

\numberwithin{equation}{section}
\numberwithin{theorem}{section}

\newcommand{\too}{\longrightarrow}

\usepackage{comment}

\newcommand{\Id}{\mathrm{Id}}

\newcommand{\BMO}{\textit{BMO}}
\newcommand{\CMO}{\textit{CMO}}

\newcommand{\CZ}{Calder\'on-Zygmund }

\makeatletter
\DeclareRobustCommand\widecheck[1]{{\mathpalette\@widecheck{#1}}}
\def\@widecheck#1#2{%
   \setbox\z@\hbox{\m@th$#1#2$}%
   \setbox\tw@\hbox{\m@th$#1%
      \widehat{%
         \vrule\@width\z@\@height\ht\z@
         \vrule\@height\z@\@width\wd\z@}$}%
   \dp\tw@-\ht\z@
   \@tempdima\ht\z@ \advance\@tempdima2\ht\tw@ \divide\@tempdima\thr@@
   \setbox\tw@\hbox{%
      \raise\@tempdima\hbox{\scalebox{1}[-1]{\lower\@tempdima\box\tw@}}}%
   {\ooalign{\box\tw@ \cr \box\z@}}}
\makeatother

\makeatletter
\@namedef{subjclassname@2020}{%
  \textup{2020} Mathematics Subject Classification}
\makeatother

\begin{document}
\baselineskip = 14pt

\title[Compactness of pseudodifferential operators]
{Symbolic calculus for a class of pseudodifferential operators with applications to compactness}

\author[\'A. B\'enyi, T. Oh, and R.H. Torres]{\'Arp\'ad B\'enyi, Tadahiro Oh, and Rodolfo H. Torres}

\address{\'Arp\'ad B\'enyi, Department of Mathematics\\
516 High St, Western Washington University\\ Bellingham, WA 98225,
USA.}

\email{benyia@wwu.edu}

\address{
Tadahiro, Oh,
School of Mathematics\\
The University of Edinburgh\\
and The Maxwell Institute for the Mathematical Sciences\\
James Clerk Maxwell Building\\
The King's Buildings\\
Peter Guthrie Tait Road\\
Edinburgh\\
EH9 3FD\\
United Kingdom}

\email{hiro.oh@ed.ac.uk}

\address{Rodolfo H. Torres, Department of Mathematics\\
University of California\\
900 University Ave., Riverside, CA 92521, USA}

\email{rodolfo.h.torres@ucr.edu}



\begin{abstract}
We prove a symbolic calculus for a class of pseudodifferential operators, and discuss its applications to $L^2$-compactness via a compact version of the $T(1)$ theorem.
\end{abstract}

\dedicatory{Dedicated to Jill Pipher in recognition of her outstanding contributions to mathematics and the mathematical profession overall}

%
\maketitle
%

\tableofcontents

\section{Introduction}

Starting with the work of Villarroya \cite{V}, there has been some interest in developing versions of the
$T(1)$ theorem of David and Journ\'e \cite{DJ} to characterize not only the boundedness of  Calder\'on-Zygmund operators but also their compactness; see, for example,  \cite{MitSto, FraGreWic, BLOT1, CLSY, BLOT2}. Unfortunately, while elegant and highly non-trivial, such characterizations have not revealed many new examples of compact Calder\'on-Zygmund operators. Perhaps the most important examples of compact Calder\'on-Zygmund operators are certain paraproducts, but their compactness can be proved by other methods,
 and in fact they are used in the proof of the characterizations of compactness.  One of the purposes of this article is to provide some new examples involving pseudodifferential operators and their commutators.

Sufficient conditions for the compactness of pseudodifferential operators on $L^2$ were originally provided by Cordes \cite{Cordes}.  However, such conditions fail to produce Calder\'on-Zygmund operators. In fact, while compact on $L^2$, Cordes' operators may not even be bounded on $L^p$ for other $p\in (1, \infty)$. In a recent work \cite{CarSorTor}, Carro, Soria and the last named author of this article combined Cordes' conditions with the typical decay of the H\"ormander class $S_{1, 0}^0$ to obtain, via a new extrapolation result, compactness of the resulting Calder\'on-Zygmund operators on weighted $L^p(w)$ spaces for all $1<p<\infty$ and all weights $w$ in the classical Muckenhoupt classes $A_p$ (whose definition will not be needed in this article)\footnote{The result can also be obtained by more general extrapolation results in \cite{coy, hl}. However, the more restrictive extrapolation in \cite{CarSorTor} has a straightforward proof and produces a stronger uniformly compact result in some appropriate sense. See \cite{CarSorTor} for further details.}. The result in \cite{CarSorTor} was presented at the conference in honor of Professor Jill Pipher to whom this article is dedicated.  Since then there has been some interest in looking at  further compactness results for pseudodifferential operators; see \cite{otropaper}.

As in the case of the classical $T(1)$ theorem, the analogous characterizations of compactness required conditions which are symmetric on the operator $T$ in question and its transpose~$T^*$.  This usually creates a non-trivial task, when the operator is presented by a symbol in pseudodifferential form and not via an explicit kernel, to verify that the class of symbols is invariant by transposition (which, as it is known, is not always the case for
general pseudodifferential operators). We will develop a symbolic calculus for the operators considered in \cite{CarSorTor}
(see Theorem \ref{THM:CPT}), and then use it to prove a new result on the compactness of commutators of a certain class of pseudodifferential operators and multiplication by bounded Lipschitz functions, using indeed the compact $T(1)$ theorem;
see Theorem \ref{THM:2}.

\section{A class of pseudodifferential operators}

We start by recalling the H\"ormander classes of symbols $S_{1, 0}^s$.
Given $s\in \mathbb R$, we say that a symbol $\sigma:\mathbb R^d\times\mathbb R^d\to\mathbb C$ belongs to the H\"ormander class $S_{1, 0}^s$, and we write $\sigma\in S_{1, 0}^s$, if it  satisfies
\[
|\partial_x^\alpha\partial_\xi^\beta \sigma (x, \xi)|\lesssim_{\alpha, \beta}(1+|\xi|)^{s-|\beta|}
\]
for all multi-indices $\alpha, \beta$. Associated with such a symbol $\sigma$, we have the pseudodifferential operator $T_\sigma$, a priori defined from $\mathcal S(\mathbb R^d)\to\mathcal S'(\mathbb R^d)$, given by
\begin{align*}
T_\sigma(f)(x)=\int_{\mathbb R^d}\sigma (x, \xi)\widehat f(\xi) e^{ix\cdot\xi} d\xi.
\end{align*}

\noi
It is a standard fact that $S_{1, 0}^0$ is closed under transposition. That is, if $\sigma\in S_{1, 0}^0$, then $(T_\sigma)^{*}=T_{\sigma^*}$ with $\sigma^{*}\in S_{1, 0}^0$; see, for example, \cite[p.\,259]{Stein} (see also  \cite[Theorem 1]{BenTor} for a bilinear version of this fact). Moreover,  if $\sigma\in S_{1, 0}^0$, then $T_\sigma$ is a Calder\'on-Zygmund operator.

In what follows, we restrict our attention to a \emph{vanishing at infinity} sub-class of $S_{1, 0}^0$ considered in \cite[Theorem 3.2]{CarSorTor}; see also \cite{Cordes}. We say that a symbol $\sigma=\sigma(x, \xi)$ belongs to the class $V_{1, 0}^s$ if it satisfies
\begin{equation}
\label{VCM}
|\partial_x^\alpha\partial_\xi^\beta\sigma (x, \xi)|\leq K_{\alpha, \beta}(x, \xi)(1+|\xi|)^{s-|\beta|}
\end{equation}

\noi
 for all multi-indices $\alpha$ and  $\beta$,  where $K_{\alpha, \beta}(x, \xi)$ is a bounded function such that $$\displaystyle\lim_{|x|+|\xi|\to\infty}K_{\alpha, \beta}(x, \xi)=0.$$

Our main result is the following symbolic calculus for the class $V_{1, 0}^s$.

\begin{theorem}
\label{THM:CPT}
If $\sigma\in V_{1, 0}^s$, then $(T_\sigma)^{*}=T_{\sigma^*}$ with $\sigma^*\in V_{1, 0}^s$. Moreover, we have the asymptotic expansion
\begin{equation}
\label{asymptotic}
\sigma^*(x, \xi)=\sum_\alpha\frac{i^{-|\al|}}{\alpha!}\partial_x^\alpha\partial_\xi^\alpha\sigma(x, -\xi),
\end{equation}
in the sense that, for every $N\in\mathbb N$,
\begin{equation}
\label{asy2}
\sigma^*(x, \xi)-\sum_{|\alpha|<N}\frac{i^{-|\al|}}{\alpha!}\partial_x^\alpha\partial_\xi^\alpha\sigma(x, -\xi)\in V_{1, 0}^{s-N}.
\end{equation}
\end{theorem}

\begin{proof}
As mentioned above, the calculus in the statement of the theorem holds with $V_{1, 0}^s$ replaced by the larger class $S_{1, 0}^s$.
An  additional complexity in the current setting is a careful accounting of the constants involved in several estimates (which depend on the spatial variable $x$ and the frequency variable $\xi$)
in  showing that they vanish at infinity.
In the following, we will  recast the computations usually done for $S_{1, 0}^s$
by  following closely other presentations in the literature. In particular, the reader may consult  \cite[pp.\,238--240 and pp.\,258--259]{Stein}, or the more elaborate presentation in \cite{BenTor} (for $s=0$) for some of the omitted details and computations.

Given  $\sigma \in V_{1, 0}^s$, it suffices to prove the result for symbols of the form:
\begin{equation}\sigma_\epsilon(x,\xi) = \sigma(x,\xi)u(\epsilon x,\epsilon \xi) \label{sigmaepsilon},
\end{equation}
where $u(x, \xi)$ is a $C^\infty$  function with compact support and $u(0, 0) = 1$,
but with estimates independent of the support of
$\sigma_\epsilon$, $\epsilon \in (0,1)$. Assuming that this is the case, a standard limiting argument removes the above truncation of the symbol.
Henceforth, and for simplicity of  notation, we assume that $\sigma(x, \xi)$ is one of the $\sigma_\epsilon(x, \xi)$ with compact support.

For $f, g\in\mathcal S$, we have
\[
\langle T_\sigma (f), g\rangle=\int_y\int_x\int_\xi \sigma (x, \xi)g(x)e^{i(x-y)\cdot\xi}d\xi dx \, f(y)dy,
\]
where the integral is absolutely convergent due to the compact support assumption on $\sigma$. It follows from here that
\[
T_\sigma^*(g)(x)=\int_y\int_\xi \sigma(y, -\xi)g(y)e^{-i(y-x)\cdot\xi}d\xi dy=:T_{[c]}(g)(x),\]

\noi
where
\begin{align}
c(y, \xi):=\sigma (y, -\xi)
\label{C1}
\end{align}

\noi
and $T_{[c]}$ is a so-called \emph{compound operator}.
It is clear that $c$ satisfies the same bound \eqref{VCM} as $\sigma$. The next step is to show that, given the compound operator $T_{[c]}$, there exists a symbol $\sigma^{*}\in V_{1, 0}^0$ such that $T_{[c]}=T_{\sigma^*}$, and also show that the asymptotic expansion \eqref{asymptotic} for $\sigma^*$ holds. The computation below is justified by the assumption on $\sigma$ having compact support. We have
\begin{align*}
T_{[c]}(g)(x)&=\iint c(y, \xi)g(y) e^{-i(y-x)\cdot\xi}d\xi dy\\
&=\int_z\int_\xi\int_y (2\pi)^{-d}c(y, \xi)e^{-i(y-x)\cdot\xi}e^{i(y-x)\cdot z}dy d\xi \, \widehat g(z)e^{ix\cdot z}dz\\
&=T_{\sigma^*}(g)(x),
\end{align*}

\noi
where
\begin{align}
\sigma^*(x, \xi)=(2\pi)^{-d}\int_z\int_y c(y, z)e^{i(x-y)\cdot (z-\xi)} dy dz.
\label{C1z}
\end{align}

\noi
Let  $\phi\in C_c^\infty$ to be an even function such that $\phi(y)=1$ for $|y|\leq 1$ and $\phi(y)=0$ for $|y|>2$. We can then further write
\begin{align}
\begin{split}
(2\pi)^d\sigma^*(x, \xi)&=\iint \phi(x-y) c(y, z)e^{i(x-y)\cdot (z-\xi)} dydz+r(x, \xi)\\
&=\iint \phi(y)c(x+y, z+\xi)e^{-iy\cdot z}dydz+r(x, \xi)\\
&=m(x, \xi)+r(x,\xi),
\end{split}
\label{transpose}
\end{align}

\noi
where
\begin{align}\label{rm}
r(x, \xi)&=\int_z\int_y (1-\phi(x-y)) c(y, z)e^{i(x-y)\cdot (z-\xi)}dydz,\\
m(x, \xi)&=\int_z \widehat b^{(2)}(x, z, z+\xi)dz.\label{mr}
\end{align}

\noi
Here,  $\widehat b^{(2)}$ denotes the Fourier transform with respect of the $y$-variable of
\begin{align}
b(x, y, \xi):=\phi(y)c(x+y, \xi).
\label{C1a}
\end{align}

\noi
 We note that $b$ satisfies similar estimates to $c$, that is
\begin{equation}
\label{b-esti}
|\partial_x^\alpha\partial_y^\tau\partial_\xi^\beta b(x, y, \xi)|
\leq A_{\alpha, \tau, \beta}(x, y, \xi)(1+|\xi|)^{s-|\beta|},
\end{equation}

\noi
where $A_{\alpha, \tau, \beta}$ is a bounded function with compact support in
all the three variables
(in particular, supported on $\{|y|\le 2\}$).

Let us first concentrate on the remainder term $r(x, \xi)$ in \eqref{rm}. Let $N_1, N_2\in\mathbb N$ with $N_2>d/2$ and $N_1\geq 2N_2$. Integration by parts gives
\[
r(x, \xi)=\int_y\int_z\frac{(\Id-\Delta_y)^{N_1}\Delta_z^{N_2}\big((1-\phi(x-y)) c(y, z)\big)}{(1+|z-\xi|^2)^{N_1}|x-y|^{2N_2}}e^{i(x-y)\cdot (z-\xi)}dzdy,
\]

\noi
where we are integrating (in $y$) over $|y-x|>1$.
By \eqref{C1} and \eqref{VCM}, we have
\[
|r(x, \xi)| \le \int_y\int_z\frac{B_{N_1, N_2}(x, y, z)(1+|z|)^{s-2N_2}}{(1+|z-\xi|)^{2N_1}(1+|x-y|)^{2N_2}}dzdy,
\]

\noi
where $B_{N_1, N_2}$ is a bounded function with compact support in $y$ and $z$.
By applying Peetre's inequality (with $k=2N_2$)
\begin{equation}
\label{Peetre}
(1+|z|)^{s-k}\leq (1+|\xi|)^{s-k}(1+|z-\xi|)^{|s|+k}, k\geq 0,
\end{equation}
and  setting $N_1=N+\lceil |s|\rceil$ and $N_2=N/2$ for  some large $N$ (which we can assume to be even),
the last integral can then  be estimated by
\begin{align}
\begin{split}
&  (1+|\xi|)^{s-N}
\int_y\int_z\frac{B_{N, N/2}(x, y, z)}{(1+|z-\xi|)^{N}(1+|x-y|)^{N}}dzdy\\
& \quad  =: (1+|\xi|)^{s-N} \wt B_{N}(x, \xi).
\end{split}
\label{C1x}
\end{align}

\noi
Note that
 the integral defining $\wt B_N(x, \xi)$ is absolutely convergent,
 with the bound  independent of the support of $\sigma$.
 Furthermore, we have
\begin{align}
\lim_{|x|+|\xi|\to\infty}\wt B_N (x, \xi)=0.
\label{C2}
\end{align}

\noi
To see this last fact, notice that $B_{N, N/2}(x, y, z)$ can be controlled by a finite sum of the form:
\begin{align}
B_{N, N/2}(x, y, z) \lesssim \sum_j  L_j(x-y)H_j(y,z),
\label{C1y}
\end{align}

\noi
where the functions $L_j$ and $H_j$ are bounded, and $H_j(y, z) \to 0$ when $|y| +|z|\to \infty$, actually independently of the support
in view of  \eqref{sigmaepsilon} and the estimates
 \eqref{VCM}
satisfied by the original symbol $\s$.
Thus, it follows from \eqref{C1x} and \eqref{C1y}
that $\wt B_N (x, \xi)$ can be controlled by a finite sum of terms of the form:
$$
\int_{|y|>1}\int_z\frac{ L_j(y)H_j(x-y, z+\xi) }{(1+|z|)^{N}(1+|y|)^{N}}dzdy
$$
to which the dominated convergence theorem (in the parameter $(x,\xi)$) can be easily applied.
\noi
Hence, for any large $N \gg 1$, we have
 \begin{align}
 |r(x, \xi)|\leq \wt B_N (x, \xi)(1+|\xi|)^{s-N}
 \label{C3}
 \end{align}

 \noi
 where $\wt B_N $ satisfies \eqref{C2}.
  A similar calculation will show that, for
  any large $N \gg1$, we also have
\begin{equation}
\label{C4}
|\partial_x^\alpha\partial_\xi^\beta r(x, \xi)|\leq \wt B_{N; \alpha, \beta}(x, \xi)(1+|\xi|)^{s-N},
\end{equation}

\noi
where
$\wt B_{N; \alpha, \beta}$
is a bounded function such that
$\wt B_{N; \alpha, \beta}(x, \xi)\to 0$
as $|x|+|\xi|\to\infty$.
In other words, \emph{the remainder term $r$ is smoothing of infinite order}.

Next, we analyze the main term $m(x, \xi)$ in \eqref{mr}. Given a multi-index $\tau$, using \eqref{b-esti}, we have
\begin{align*}
|z^\tau \widehat b^{(2)}(x, z, z+\xi)|&=
\bigg|\int_y b(x, y, z+\xi)\partial_y^{\tau}(e^{-iy\cdot z})dy\bigg|\\
& =\bigg|\int_y (\partial_y^{\tau}b(x, y, z+\xi))e^{-iy\cdot z}dy\bigg|
\leq \int_y |\partial_y^{\tau}b(x, y, z+\xi)|dy\\
& \leq (1+|z+\xi|)^s \int_y A_{0, \tau, 0}(x, y, z+\xi)dy. 
\end{align*}

\noi
Thus, for any $N\in\mathbb N$, we obtain
\[
|\widehat b^{(2)}(x, z, z+\xi)|\leq A'_N(x, z+\xi)(1+|z+\xi|)^s (1+|z|)^{-N},
\]

\noi
where $A'_N$ is a bounded function such that $A'_{N}(x, z+\xi)\to 0$ as $|x|+|z+\xi|\to \infty$.
By a similar computation,  for any multi-indices $\alpha, \beta$,  we have
\begin{equation}
\label{b-derivatives}
|\partial_x^\alpha\partial_\xi^\beta\widehat b^{(2)}(x, z, z+\xi)|\leq A'_{N; \alpha, \beta}(x, z\xi)(1+|z+\xi|)^{s-|\beta|}(1+|z|)^{-N},
\end{equation}

\noi
where $A'_{N; \alpha, \beta}$
is a  bounded function such that
\begin{align}
\lim_{|x|+|\xi|\to \infty}A'_{N; \alpha, \beta}(x, \xi)= 0.
\label{C5}
\end{align}

\noi
We note from \eqref{C1a}
 that the functions $A'_{N; \alpha, \beta}$ do not depend on the support of $\sigma$ as they are obtained by integrating over $\{|y|\le 2\}$.
From \eqref{mr},
\eqref{b-derivatives}, and
  an appropriate version of \eqref{Peetre}, namely
$$(1+|z+\xi|)^{s-|\beta|}\leq (1+|\xi|)^{s-|\beta|}(1+|z|)^{|s|+|\beta|},$$
we have
\begin{align}
\begin{split}
|\partial_x^\alpha\partial_\xi^\beta m(x, \xi)|&\leq \int A'_{N; \alpha, \beta}(x, z+\xi)(1+|z+\xi|)^{s-|\beta|}(1+|z|)^{-N}dz\\
&\leq (1+|\xi|)^{s-|\beta|}\int A'_{N; \alpha, \beta}(x, z+\xi)(1+|z|)^{-N+|s|+|\beta|}dz.
\end{split}
\label{C6}
\end{align}

\noi
Now, let
\begin{align}
A''_{\alpha, \beta}(x, \xi):=\int A'_{d+1+\lceil |s|\rceil+|\beta|; \alpha, \beta}(x, z+\xi)(1+|z|)^{-d-1} dz.
\label{C7}
\end{align}

\noi
Then, from \eqref{C5}
and
the  dominated convergence theorem,
 we see that $A''_{\alpha, \beta}$ is a bounded function such that
$A''_{\alpha, \beta}(x, \xi)\to 0$ as $|x|+|\xi|\to\infty$.
Hence, from \eqref{C6} with
(with $N=d+1+\lceil |s|\rceil+|\beta|$) and \eqref{C7},
we obtain
\begin{equation}
\label{m-derivatives}
|\partial_x^\alpha\partial_\xi^\beta m(x, \xi)|\leq A''_{\alpha, \beta}(x, \xi)(1+|\xi|)^{s-|\beta|},
\end{equation}

\noi
where $A''_{\alpha, \beta}$ is a bounded function, vanishing at infinity,
and is   independent of the support of the symbol $\s$.
Therefore, from \eqref{C4} and \eqref{m-derivatives},
we conclude that $\sigma^*$ defined in \eqref{C1z} is in $V_{1, 0}^s$, with estimates independent of the support.

Next, we show that $\sigma^*$ has
the asymptotic expansion \eqref{asymptotic}
in the sense of \eqref{asy2}.
On the one hand, given  $M, N\in\mathbb N$,
 by applying the Taylor remainder theorem
to  $\widehat{b}^{(2)}(x, z, z+\xi)$
in  the third argument,
  \eqref{b-derivatives},
  and  Peetre's inequality, we have
\begin{align}
\begin{split}
&\bigg|\widehat b^{(2)} (x, z, z+\xi)- \sum_{|\alpha|<N}\partial_\xi^\alpha\widehat b^{(2)}(x, z, \xi)\frac{z^\alpha}{\alpha!}\bigg|\\
& \quad
= |z|^N
\bigg|\sum_{|\al| = N}\frac N{\al!}
\int_0^1 (1-t)^{N-1}\partial_\xi^\al \widehat b^{(2)}(x, z, tz+\xi)dt\bigg|\\
& \quad  \lesssim (1+|z|)^{N-M}
\sum_{|\al| = N}
\int_0^1
A'_{M; 0, \al}(x, tz+\xi)(1+|tz+\xi|)^{s-N}dt\\
& \quad
\lesssim (1+|\xi|)^{s-N}(1+|z|)^{|s|+2N-M}\wt A_{M;  N}(x, z, \xi),
\end{split}
\label{estimates-b}
\end{align}

\noi
where
\begin{align}
\wt A_{M;  N}(x, z, \xi) =
\sum_{|\al| = N}
\int_0^1
A'_{M; 0, \al}(x, tz+\xi)dt.
\label{C7a}
\end{align}

\noi
Note that,  for fixed $z$,
the condition  $|x|+|\xi|\gg |z|$
implies
 $|x|+|tz+\xi|\sim |x|+|\xi|$ uniformly in $0 \le t \le 1$.
 Then, from \eqref{C7a} and \eqref{C5},
 we see that
$\wt A_{M;  N}$ is a bounded function such that
\begin{align}
\lim_{|x|+|\xi|\to\infty} \wt A_{M;  N}(x, z, \xi) = 0
\label{C7b}
\end{align}

\noi
for each fixed $z \in \R^d$.

On the other hand,
from  \eqref{C1}, the fact that $\phi \equiv 1$ in the neighborhood of $y = 0$,
and \eqref{C1a},
we have
\begin{align}
\begin{split}
\dx^\al \s(x, -\xi)&  =
\dd_y^\al \big(\phi(y)c(x+y, \xi)\big)|_{y = 0}
= \dd_y^\al b(x, 0, \xi)\\
& = \frac {{i}^{|\al|}}{(2\pi)^d} \int
z^\al
 \ft b^{(2)}(x, z, \xi)  dz.
\end{split}
\label{C8}
\end{align}

\noi
Hence, from  \eqref{transpose}, \eqref{mr},
\eqref{C8},
 \eqref{estimates-b} with $M=2N+\lceil |s|\rceil+d+1$,
 and \eqref{C3},
  we obtain
\begin{align*}
&\bigg|\sigma^*(x, \xi)
-\sum_{|\alpha|<N}\frac{i^{-|\al|}}{\alpha!}\partial_x^\alpha\partial_\xi^\alpha\sigma(x, -\xi)\bigg|\\
&\lesssim \int\bigg|\widehat b^{(2)} (x, z, z+\xi)- \sum_{|\alpha|<N}\partial_\xi^\alpha\widehat b^2(x, z, \xi)\frac{z^\alpha}{\alpha!}\bigg|dz +|r(x, \xi)|\\
&\lesssim\widetilde K_N(x, \xi) (1+|\xi|)^{s-N},
\end{align*}
where
\[
\wt  K_N(x, \xi)=
\wt B_N(x, \xi)+
\int \wt A_{2N+\lceil |s|\rceil +d+1; N}(x, z, \xi)(1+|z|)^{-d-1}dz.
\]

\noi
From \eqref{C2}, \eqref{C7b},
and  the  dominated convergence theorem, we see that
$\widetilde K_N (x, \xi)$ is a bounded function such that
\(\widetilde K_N (x, \xi)\to 0\) as $|x|+|\xi|\to\infty$.
A similar computation
yields an analogous bound
on
\[ \dx^\al \dd_\xi^\be \bigg(\sigma^*(x, \xi)
-\sum_{|\alpha|<N}\frac{i^{-|\al|}}{\alpha!}\partial_x^\alpha\partial_\xi^\alpha\sigma(x, -\xi)\bigg), \]
thus establishing \eqref{asy2}.
\end{proof}


\begin{remark}\rm
The conditions on the symbol assumed by Cordes \cite{Cordes} (corresponding to the case $s=0$ in \eqref{VCM}) are only that $\sigma$  has continuous and bounded derivatives  $\partial^\alpha_x \partial^\beta_\xi \sigma$ for all multi-indices $|\alpha|, |\beta| \leq 2N$, where
$N= \big[\frac d2\big] +1$, and that
\begin{equation}\label{finitederivatives}
(\Id-\Delta_x)^N (\Id-\Delta_\xi)^N \sigma(x,\xi) \too 0,
\end{equation}

\noi
as $|x| + |\xi| \to \infty$.   However, as mentioned already, the corresponding operator cannot be Calder\'on-Zygmund (see the definition in the next section) since it may not be even bounded on $L^p$ for $p\neq2$. If one assumes, in addition to \eqref{finitederivatives}, that the symbol belongs to $S^0_{1,0}$, then the corresponding pseudodifferential operator becomes bounded on all weighted spaces $L^p(w)$, where $1<p<\infty$ and $w\in A_p$, and by the extrapolation results in \cite{coy, hl},
 it can be proved to be compact. Likewise, in the main result in \cite{CarSorTor}, the conditions in \eqref{VCM} need to be assumed only for a (large) finite number of derivatives.  These conditions on the symbol $\sigma$ are hence weaker than $\sigma$ belonging to $V_{1, 0}^0$, for which we impose \eqref{VCM} for
 {\it all} the derivatives of $\sigma$
 in order for
 the symbolic calculus, and in particular, for the expansion formula \eqref{asy2}, to hold.
\end{remark}

\begin{remark}\rm
Symbols in the class $V_{1, 0}^0$ are very easy to construct.  Trivially, if we let
$$
\sigma(x,\xi) = m(x)\psi(\xi),
$$
where $m,  \psi \in C^\infty$, $|\partial^\alpha m(x)| \to 0$ as $|x| \to \infty$ and $\psi$ has compact support, then
$
\sigma(x,\xi) \in V^{-N}_{1,0}
$
for all $N$.

More interestingly,  one can consider the following modification of the \emph{elementary symbols \`a la Coifman-Meyer}; see \cite[p.\,41]{CM}. Given $m_j,  \psi \in C^\infty$ such that $|\partial^\alpha m_j(x)| \leq C_\alpha$ for all $x$, $j$ and $\alpha$,  $\partial^\alpha m_j(x) \to 0$ as $|x|\to \infty$ for each $\alpha$ and uniformly in $j$, and $ \psi$ compactly supported on $ \big\{\frac 12<|\xi|<2\big\}$, define
\begin{equation}
\label{elementary}
\sigma(x, \xi)= \sum_{j=0}^\infty  m_j(x) 2^{-j}\psi(2^{-j}\xi).
\end{equation}
Due to the condition on the support of $\psi$, for a fixed $(x, \xi)$ there are only finitely many terms in the summation above. More precisely, for a fixed $j_0$ and $2^{j_0-1}<|\xi|<2^{j_0+1}$, we have
\begin{align*}
|\partial^\alpha_x \partial^\beta_\xi \sigma(x,\xi)| &\lesssim  \sum_{k=-1}^1 |\partial^\alpha_x m_{j_0+k}(x)| 2^{-j_0-k}|\partial^\beta_\xi \psi(2^{-j_0-k}\xi) |\\
&\lesssim_{\beta}  2^{-j_0 (|\beta|+1)}\sum_{k=-1}^1|\partial^\alpha_x m_{j_0+k}(x)|\\
&\lesssim_{\beta}  (1+|\xi|)^{-(|\beta|+1)}\sum_{k=-1}^1\sup_{j}|\partial^\alpha_x m_{j+k}(x)|\\
&\lesssim K_{\alpha, \beta}(x, \xi)(1+|\xi|)^{-|\beta|},
\end{align*}
where $K_{\alpha, \beta}(x, \xi) \to 0$ as $|x|+|\xi| \to \infty.$ Analogous to the case of symbols in $S_{1, 0}^0$, it may be possible to write every symbol in $V_{1, 0}^0$ as some combination of elementary symbols of the
form~\eqref{elementary}. Interested readers may consider adapting some of the methods suggested in
\cite[Lemma~1]{BenTor} to see if such a result holds in this context.
\end{remark}

\section{Compact \(T(1) \) theorems}

In this section, we collect  several definitions and compactness versions of the $T(1)$ theorem.
\begin{definition}\rm
A continuous, linear operator $T: \S(\mathbb R^d) \to \S'(\mathbb R^d)$ is said to be a singular integral operator with a Calder\'on-Zygmund kernel $K$ if the restriction of its distributional Schwartz kernel to the set $\{(x,y)\in \mathbb R^{2d}: x\neq y\}$ agrees with a locally integrable function $K$ satisfying
\begin{equation}\label{regularity}
|\partial^{\beta}K(x,y)|\leq C_K|x-y|^{-d-|\beta|}, \quad |\beta|\leq 1.
\end{equation}
In particular, we have
$$
Tf(x)=\int_{\mathbb R^d}  K(x,y)f(y)  dy
$$
for $x\notin \supp  f$.
\end{definition}

While weaker regularity conditions than  \eqref{regularity} can be assumed, in the case of the pseudodifferential operators we consider in this article, the condition \eqref{regularity}
is actually satisfied for derivatives of any order $|\beta| \in \mathbb N$
(but with constants depending on $|\beta|$).
A singular integral operator with a Calder\'on-Zygmund kernel is said to be a Calder\'on-Zygmund operator,
 if it extends as  a bounded  linear operator $T:L^2(\mathbb R^d) \to L^2(\mathbb R^d)$. In such a case, standard results of the Calder\'on-Zygmund theory imply that the operator $T$ is also bounded on $L^p (\mathbb R^d)$ for $1<p<\infty$. Appropriate end-point and weighted estimates
 are also known to hold for such $T$.

As usual, we write $C^\infty_c(\mathbb R^d)$ for the space of $C^\infty$-functions with compact support, and $C_0(\mathbb R^d)$ for the space of continuous functions vanishing at infinity.

\begin{definition}\rm
We say that a function $\phi \in C_c^\infty (\mathbb R^d)$ is
a \emph{normalized bump function}  of order $M$
if $\supp \phi \subset B_0(1)$
and $\|\dd^\al \phi\|_{L^\infty} \leq 1$
for all multi-indices $\al$ with $|\al| \leq M$.
Here, given $r > 0$ and $x \in \R^d$,
  $B_x(r)$ denotes the ball of radius $r$ centered at $x$.

\end{definition}

Given $x_0 \in \R^d$ and $R>0$, we set
\begin{align}
\phi^{x_0, R}(x) = \phi\bigg(\frac{x-x_0}{R}\bigg).
\label{scaling}
\end{align}

\noi
The statement of the $T(1)$ theorem of David and Journ\'e \cite{DJ} is the following.

\begin{oldtheorem}[$T(1)$ theorem]\label{THM:DJ}
Let $T: \S (\R^d)\to \S'(\R^d)$ be a linear singular integral  operator
with a 
Calder\'on-Zygmund kernel.
Then, $T$ can be extended to a bounded operator on $L^2 (\R^d)$
if and only if
it satisfies the following two conditions\textup{:}

\smallskip

\begin{itemize}
\item[\textup{(A.i)}]
$ T $ satisfies the weak boundedness property\textup{;}
there exists $M \in \mathbb N\cup \{0\}$
such that we have
\begin{align}
 \big| \big\langle T(\phi_1^{x_1, R}), \phi_2^{x_2, R}\big\rangle \big|\les R^d
\label{WBP1}
 \end{align}

\noi
for any normalized bump functions $\phi_1$ and $\phi_2$
of order $M$, $x_1, x_2 \in \R^d$,
and $R>0$.

\smallskip
\item[\textup{(A.ii)}]
$ T(1)$
and $T^*(1)$
are in $\BMO$\,.
\end{itemize}

\end{oldtheorem}

\begin{definition} \label{DEF:WCP1} \rm

We say that 	
a linear singular integral  operator  $T: \S(\R^d) \to \S'(\R^d)$ with a Calder\'on-Zygmund kernel
satisfies the {\it weak compactness property}
if there exists $M \in \mathbb N\cup \{0\}$
such that we have
\begin{align*}
\lim_{|x_0|  + R + R^{-1} \to \infty}
R^{-d} \big| \big\langle T(\phi_1^{x_0 + x_1, R}), \phi_2^{x_0 + x_2, R}\big\rangle \big|
= 0
 \end{align*}

\noi
for any normalized bump functions $\phi_1$ and $\phi_2$
of order $M$ and  $x_1, x_2 \in \R^d$.

\end{definition}

In order to state the compact counterparts of the $T(1)$ theorem,
we also need to recall  the following important subspace of $\BMO$\,.

\begin{definition}\label{DEF:CMO}\rm
The closure of $C_c^\infty(\R^d)$ in the $\BMO$\,~topology is called the space of functions of \emph{continuous mean oscillation}, and it is denoted by $\CMO$\,.
  It  also holds that $\CMO = \overline{C_0(\R^d)}^{\BMO}$;
see \cite[Th\'eor\`eme 7]{Bour}.
\end{definition}

The following version of the compact $T(1)$ theorem was recently proved in \cite{MitSto};
see also~\cite{V}.
For a version of a compact $T(b)$ theorem, see \cite{V2}.

\begin{oldtheorem}[compact $T(1)$ theorem]\label{THM:MS}
Let $T$ be as in Theorem \ref{THM:DJ}.
Then, $T$ can be extended to a compact operator on $L^2$
if and only if
it satisfies the following two conditions\textup{:}

\smallskip

\begin{itemize}
\item[\textup{(B.i)}]
$ T $ satisfies the weak compactness  property.

\smallskip
\item[\textup{(B.ii)}]
$ T(1)$
and $T^*(1)$
are in $\CMO$\,.
\end{itemize}

\end{oldtheorem}

In the following, we say that  a linear singular integral  operator $T$
with a
Calder\'on-Zygmund kernel
is a {\it compact \CZ operator} if it is compact on $L^2$.

\begin{remark}\label{REM:CMO} \rm
 If $T$ is a compact \CZ operator,
then it is also compact in the following settings:

\smallskip

\begin{itemize}
\item[(a)]
On
$L^p$ for any $1 < p < \infty$,
this follows from interpolation
of compactness
 \cite{Kra}.

\smallskip

\item[(b)]
From $L^\infty$ to $\CMO$\,; see
\cite{PPV}.
.

\smallskip

\item[(c)]
From the Hardy space $H^1$ to $L^1$;
 see~\cite[p.\,1285]{OV}.

\smallskip

\item[(d)]
From $L^1$ to $L^{1, \infty}$; see~\cite{OV}.

\smallskip

\item[(e)]
 From $\BMO$ to $\CMO$
 under the cancellation assumption that $T(1) = T^*(1) = 0$;  see~\cite{PPV}.

\smallskip

\item[(f)]
 From $ H^1$ to $H^1$
 under the cancellation assumption that $T(1) = T^*(1) = 0$; this follows from (e) by duality, but see also \cite{BLOT2}.

\end{itemize}
\end{remark}

\begin{remark} \rm
Our definitions of the weak boundedness property \eqref{WBP1}
and the weak compactness property (Definition \ref{DEF:WCP1})
are slightly different from those in \cite{DJ, MitSto},
but one can easily check that they are equivalent.
\end{remark}

As in the boundedness case
(see \cite[Theorem 3 on p.\,294]{Stein} and
also \cite{GT, BO2}),
 it is possible to state a version of the compact $T(1)$ theorem without referring to $T(1)$ at all. The following version was established in \cite{BLOT2}.

\begin{oldtheorem}[compact $T(1)$ theorem \`a la Stein]\label{THM:Stein}
Let $T$ be as in Theorem \ref{THM:DJ}.
Then, $T$ can be extended to a compact operator
on $L^2$
if and only if
the following two conditions hold\textup{:}

\begin{itemize}
\item[(i)]
\textup{($L^2$-condition).}
There exists $M \in \mathbb{N} \cup\{0\}$
such that
\begin{align*}
\lim_{|x_0| + R + R^{-1} \to \infty}
\bigg(
R^{-\frac d2}  \|T(\phi^{x_0, R})\|_{L^2}
+
R^{-\frac d2}   \|T^*(\phi^{x_0, R})\|_{L^2}
\bigg) = 0
\end{align*}

\noi
 holds for any normalized bump function $\phi$
of order $M$.

\smallskip

\item[(ii)]
\textup{($L^\infty$-condition).}
Given any
weak-$\ast$ convergent sequence $\{f_n\}_{n \in \N}$
in $L^\infty$
such that  $\{T(f_{n})\}_{n \in \N}$
and $\{T^*(f_{n})\}_{n \in \N}$
are bounded in $\CMO$,
both the sequences
  $\{T(f_{n})\}_{n \in \N}$
and $\{T^*(f_{n})\}_{n \in \N}$
are precompact in $\CMO$\,.

\end{itemize}

\end{oldtheorem}

\section{An application}

In this section, we present  an application
of our symbolic calculus and the compact $T(1)$ theorem.

We first note that for a pseudodifferential operator with symbol $\s$,
 we have $T_\sigma (1)=\sigma (x, 0)$. Thus, if $\sigma\in V_{1, 0}^0$,
then we have
$$
| \sigma (x, 0)|\leq K_\alpha (x, 0) \too 0,
$$

\noi
as $|x| \to \infty$.
Moreover,
since
$\sigma\in S_{1, 0}^0$,  $T_\sigma$ is a Calder\'on-Zygmund operator
and thus it follows from Theorem \ref{THM:DJ}
that $T_\s(1) \in \BMO$.
In particular,
this implies that $T_\sigma (1) = \sigma (x, 0)\in \CMO = \overline{C_0}^{\BMO}$.
In view of  Theorem \ref{THM:CPT},
we  have $\s^* \in V^0_{1, 0}$
and thus
we also obtain $T_\sigma^* (1)
= T_{\sigma^*} (1) \in\CMO$.

We already know
from \cite[Theorem 3.2]{CarSorTor}
that the class $V^0_{1, 0}$ actually produces compact operators on $L^p(w)$ for all  $1<p<\infty$ and $w\in A_p$.
We will use the symbolic calculus in a crucial way to show the compactness of the commutators of
pseudodifferential operators with symbols in $V_{1, 0}^1$ with pointwise multiplication by bounded Lipschitz functions.

The commutator $[T_\sigma, M_a]$
of a pseudodifferential operator $T_\sigma$ with the multiplication $M_a$ by the function $a$ is given by
$$
[T_\sigma, M_a](f) =(T_\sigma M_a-M_aT_\sigma)(f) = T_\sigma (af)-aT_\sigma(f).
$$

\noi
Recall that if $\sigma \in S_{1, 0}^1$ and $a \in \BMO$ then $[T_\sigma, M_a]$ is bounded on $L^2$, while if $a\in \CMO$ then $[T_\sigma, M_a]$ is compact on $L^2$; see \cite{CRW} and \cite{U}, respectively.  On the other hand, if $\sigma \in V^0_{1,0}$ and $a\in L^\infty$, then it is trivial to see that $[T_\sigma, M_a]$ is a compact operator on $L^2$ because the class of compact operators is a two-sided ideal of the algebra of bounded operators.

Let now $\s\in V_{1, 0}^1$ and $a$ be a bounded Lipschitz function, that is $a\in L^\infty$ and $\nabla a\in L^\infty$.
Since we have $V_{1, 0}^1\subset S_{1, 0}^1$,
it follows from \cite[Theorem 4 on p.\,90]{CoiMey} that  the commutator
$[T_\sigma, M_a]$ is a linear Calder\'on-Zygmund operator,
 bounded on $L^p$ for any $1<p<\infty$.
 In the following,
we show that this boundedness can be improved to compactness for symbols $\sigma\in V_{1, 0}^1$.

\begin{theorem}
\label{THM:2}
Let $\sigma\in V_{1, 0}^1$ and let $a$ be a bounded Lipschitz function.
 Then, $[T_\sigma, M_a]$ is a compact
\CZ operator.
\end{theorem}

\begin{proof}

We will follow  the argument in \cite{BenOh}
for boundedness of commutators in the bilinear setting.
We break the proof of Theorem \ref{THM:2} into several steps.

\smallskip

\noi
$\bullet$ {\it Step 1:}
{\it Reduction from $V_{1, 0}^1$ to $V_{1, 0}^0$}.

 Let $\widetilde{\sigma}(x,\xi)=\sigma (x, 0) \psi(\xi)$,
 where $\psi \in C^\infty_c$ and $\psi(0)=1$. Then, $\widetilde{\sigma} \in V^{-N}_{1,0}$ for any $N\geq 0$.
Define  $\s_0=\s - \wt \s \in V^{1}_{1,0}$ such that
$T_{\s}=T_{\s_0}+T_{\wt \s}$. 
Since $T_{\wt \s}$ is compact on $L^2$ and $a\in L^\infty$, so is the commutator
$[T_{\wt \sigma}, M_a].$ Moreover, since $\wt \sigma \in S^1_{1,0}$, we have that $[T_{\wt \sigma}, M_a]$ is also a compact Calder\'on-Zygmund operator.
Hence, by Theorem \ref{THM:MS},
we have
\begin{equation}
[T_{\wt \sigma}, M_a](1) \in \CMO.
\label{added}
\end{equation}

Next, by the fundamental theorem of calculus
and the fact that  $\sigma_0 (x, 0)=0$,
we have
\begin{equation}
\label{repres1}
\sigma_0 (x, \xi)=\sigma_0 (x, \xi)-\sigma_0 (x, 0)=\sum_{j=1}^d \xi_j\sigma_j (x, \xi),
\end{equation}
where
\[
\sigma_j(x, \xi)=\int_0^1 \partial_{\xi_j'}\sigma_0 (x, \xi')|_{\xi'=t\xi}\,dt.
\]

\noi
We claim that $\sigma_j\in V_{1, 0}^0$ for each $j =1, \dots, d$.
Noting that
\[
t(1+  t|\xi|)^{-1}\leq (1+|\xi|)^{-1},
\]

\noi
 for any $0 \le t\le 1$
 and applying
\eqref{VCM}, we have
\begin{align*}
|\partial_x^\alpha\partial_\xi^\beta\sigma_j (x, \xi)|
&=\bigg|\int_0^1 t^{|\beta|}\partial_x^\alpha\partial_{\xi}^\beta\partial_{\xi_j'}\sigma_0 (x, \xi')|_{\xi'=t\xi}\,dt\bigg|\\
&\leq \int_0^1 t^{|\beta|}K_{\alpha, \beta+1}(x, t\xi)(1+t|\xi|)^{-|\beta|}dt\\
&\leq (1+|\xi|)^{-|\beta|}\int_0^1 K_{\alpha, \beta+1}(x, t\xi)dt\\
&=:(1+|\xi|)^{-|\beta|}\widetilde{K}_{\alpha, \beta}(x, \xi).
\end{align*}

\noi
Note that, since $K_{\alpha, \beta+1}$ is a bounded function, so is $\wt K_{\alpha, \beta}$.
Moreover, for each $t>0$, 
$K_{\alpha, \beta+1}(x, t\xi)$ tends to $0$ as $|x|+|\xi|\to \infty$.
Therefore, by the dominated convergence theorem
we conclude that  $\wt K_{\alpha, \beta}(x, \xi)$ tends to  $0$ as $|x|+|\xi|\to\infty$.
This proves $\sigma_j\in V_{1, 0}^0$.

From \eqref{repres1}, we have
\begin{equation}
\label{repres2}
T_{\sigma_0} (f)= - i \sum_{j=1}^d T_{\sigma_j}(\partial_{j}f).
\end{equation}

\noi
Hence, we obtain
\begin{equation}
[T_{\sigma_0}, M_a](f)=
-i \sum_{j=1}^d
\Big(T_{\sigma_j}(f\partial_{j}a+a\partial_{j}f)-aT_{\sigma_j}(\partial_{j}f)\Big).
\label{repres3}
\end{equation}

\smallskip

\noi
$\bullet$ {\it Step 2:
The commutator and its transpose on the function 1.}

From  \eqref{repres3}, we have
\begin{align}
[T_{\sigma_0}, M_a](1)
=-i \sum_{j=1}^d T_{\sigma_j}(\partial_{j}a).
\label{X1}
\end{align}

\noi
Since $\sigma_j\in V_{1, 0}^0$,
it follows
that $T_{\sigma_j}$ is a compact Calder\'on-Zygmund operator
and, by Remark~\ref{REM:CMO}\,(b), $T_{\sigma_j}: L^\infty\to\CMO$\,.
Hence, recalling that  $\nb a\in L^\infty$,
we obtain from~\eqref{X1} that
\begin{equation}
\label{T1}
[T_{\sigma_0}, M_a](1)\in\CMO\,.
\end{equation}

\noi
Therefore, from
  \eqref{added} and \eqref{T1} with $\s = \s_0 + \wt \s$, we conclude that
\begin{align}
  [T_{\sigma}, M_a](1)\in\CMO.
\label{T1a}
\end{align}

A simple calculation shows that the transpose of the commutator satisfies
\[
[T_\sigma, M_a]^*=-[T_\sigma^*, M_a]=-[T_{\sigma^*}, M_a],
\]
where $\sigma^*\in V_{1, 0}^1$; see Theorem \ref{THM:CPT}.
Therefore, repeating the calculation above by replacing $\sigma$ by $\sigma^*$, we also obtain
\begin{equation}
\label{T*1}
[T_\sigma, M_a]^*(1)\in\CMO\,.
\end{equation}

\smallskip

\noi
$\bullet$ {\it Step 3:
Weak compactness property.}

From Steps 1 and 2, we may assume that $\s(x,0)= 0$ so the formula \eqref{repres1} holds. Fix $x_1, x_2\in\R^d$ and two normalized bump functions $\phi_1$ and $ \phi_2$.
In this part, we prove the weak compactness property
of $[T_\sigma, M_a]$, namely,
\begin{align}
\lim_{|x_0|  + R + R^{-1} \to \infty}
R^{-d} \big| \big\langle [T_\sigma, M_a](\phi_1^{x_0 + x_1, R}), \phi_2^{x_0 + x_2, R}\big\rangle \big|
= 0.
\label{X2}
 \end{align}

\noi
In the following,
we assume that  $\phi_1$ and $ \phi_2$
are normalized  bump functions of order at least $M \ge 1$.
The $M= 0$ case follows from \eqref{X2} for $M \ge 1$
and an approximation argument.

First,  note that
$[T_\sigma, M_a] =[T_\sigma, M_{\tilde{a}}]$, where $\widetilde{a}=a-a(x_0+x_1)$.
Thus, for each given $x_0 \in \R^d$,
we may assume without loss of generality that
 $a(x_0+x_1)=0$.
Then, from the fundamental theorem of calculus, we have
\begin{equation}
\label{X3}
\|a\|_{L^\infty(B_{x_0+x_1}(R))}\lesssim R\|\nabla a\|_{L^\infty}.
\end{equation}

\noi
Similarly,
we  have
\begin{equation}
\label{X3a}
\|a\|_{L^\infty(B_{x_0+x_2}(R))}\lesssim (|x_1 - x_2| + R)\|\nabla a\|_{L^\infty}
\les  R\|\nabla a\|_{L^\infty}
\end{equation}

\noi
for any $R \ges |x_1 - x_2|$.
Recall also from \eqref{scaling} that
\begin{equation}
\label{X4}
\|\partial_x^\alpha \phi_j^{x_0+x_j, R}\|_{L^2}\lesssim R^{\frac{d}{2}-|\alpha|},
\end{equation}

\noi
uniformly in  $x_0, x_j \in \R^d$, $j = 1, 2$.

Given  $x_0 \in \R^d$ and $R> 0$ (recall that $x_1$ and $x_2$ are fixed),
set
\begin{align*}
A(x_0, R)&:=R^{-d}\left|\big\langle T_\sigma (a\phi_1^{x_0+x_1, R}), \phi_2^{x_0+x_2, R}\big\rangle\right|,\\
B(x_0, R)&:=R^{-d}\left|\big\langle aT_\sigma (\phi_1^{x_0+x_1, R}), \phi_2^{x_0+x_2, R}\big\rangle\right|.
\end{align*}

\noi
Then,
we have
\begin{align}
R^{-d}\left|\langle [T_\sigma, M_a](\phi_1^{x_0+x_1, R}), \phi_2^{x_0+x_2, R}\rangle\right|\leq A(x_0, R)+B(x_0, R).
\label{X4a}
\end{align}

\noi
Let $R \ges |x_1 - x_2|$.
Then, from \eqref{X4}, \eqref{X3a}, and \eqref{repres2}, we have
\begin{align}
\begin{split}
B(x_0, R)&\leq R^{-d}\|aT_\sigma (\phi_1^{x_0+x_1, R})\|_{L^2(B_{x_0+x_2}(R))}\|\phi_2^{x_0+x_2, R}\|_{L^2}\\
&\lesssim R^{-\frac{d}{2}}\|a\|_{L^\infty(B_{x_0+x_2}(R))}\|T_\sigma (\phi_1^{x_0+x_1, R})\|_{L^2(B_{x_0+x_2}(R))}\\
&\lesssim
R^{-\frac{d}{2}+1}\|\nabla a\|_{L^\infty}\sum_{j=1}^d\|T_{\sigma_j}(\partial_j\phi_1^{x_0+x_1, R})\|_{L^2}\\
&\lesssim
 \|\nabla a\|_{L^\infty}\sum_{j=1}^d R^{-\frac{d}{2}}\|T_{\sigma_j}(\phi_{1, j}^{x_0+x_1, R})\|_{L^2},
\end{split}
\label{X5}
\end{align}

\noi
where $\phi_{1, j}=\partial_j\phi_1$.
Recall now from Step 1
that $T_{\sigma_j}$ is a compact \CZ operator.
 Then, by noting that $\phi_{1, j}$ is a normalized  bump function (of order $M-1$),
 it follows from
 Theorem~\ref{THM:Stein} (necessity of the $L^2$-condition)
 that
\begin{align}
\lim_{|x_0|+R+R^{-1}\to\infty}R^{-\frac{d}{2}}\big\|T_{\sigma_j}(\phi_{1, j}^{x_0+x_1, R})\big\|_{L^2}=0
\label{X6}
\end{align}

\noi
for any $j = 1, \dots, d$.
Hence, from \eqref{X5} and \eqref{X6}, we obtain
\begin{align}
\lim_{|x_0|+R+R^{-1}\to\infty}B(x_0, R)=0.
\label{X7}
\end{align}

\noi
Proceeding as in \eqref{X5}, we have
\begin{align*}
A(x_0, R)&\leq R^{-d}\|T_\sigma (a\phi_1^{x_0+x_1, R})\|_{L^2}\|\phi_2^{x_0+x_2, R}\|_{L^2}\\
&\lesssim R^{-\frac{d}{2}}\sum_{j=1}^d\|T_{\sigma_j}(\partial_j(a\phi_1^{x_0+x_1, R})\|_{L^2}\\
&\lesssim \sum_{j=1}^d \|T_{\sigma_j}(\partial_j a  \cdot R^{-\frac{d}{2}}  \phi_1^{x_0+x_1, R})\|_{L^2}\\
& \quad +
\sum_{j=1}^d
\|T_{\sigma_j}(R^{ -1}  \ind_{B_{x_0+x_1}(R)} a \cdot R^{-\frac d2} \phi_{1, j}^{x_0+x_1, R})\|_{L^2},
\end{align*}

\noi
where, in the last step, we used the chain rule:
$\partial_j\phi_1^{x_0+x_1, R} = R^{-1}(\partial_j\phi_1)^{x_0+x_1, R}
= R^{-1} \phi_{1, j}^{x_0+x_1, R}$
and the fact that $\phi_{1, j}^{x_0+x_1, R}$ is supported on the ball $B_{x_0+x_1}(R)$.
Since  $\partial_j a\in L^\infty$, we observe that
$\partial_j a  \cdot R^{-\frac{d}{2}}  \phi_1^{x_0+x_1, R}$ converge weakly to $0$ in $L^2$ as $|x_0| + R + R^{-1} \to \infty$.
From \eqref{X3}, we have $R^{ -1}  \ind_{B_{x_0+x_1}(R)} a \in L^\infty$
and thus we see that
$R^{ -1}  \ind_{B_{x_0+x_1}(R)} a \cdot R^{-\frac d2} \phi_{1, j}^{x_0+x_1, R}$
also
converge weakly to $0$ in $L^2$ as $|x_0| + R + R^{-1} \to \infty$.
Hence, since $T_{\sigma_j}$ is compact on $L^2$, we obtain
\begin{align}
\lim_{|x_0|+R+R^{-1}\to\infty}A(x_0, R)=0.
\label{X8}
\end{align}

\noi
Therefore, from \eqref{X4a}, \eqref{X7}, and \eqref{X8}, we obtain
\[
\lim_{|x_0|+R+R^{-1}\to\infty}R^{-d}\left|\langle T(\phi_1^{x_0+x_1, R}), \phi_2^{x_0+x_2, R}\rangle\right|=0,
\]
which proves the weak compactness property of the commutator $[T_\sigma, M_a]$.

\smallskip

Finally, from  Theorem \ref{THM:MS} with
Steps 2 and 3,
we conclude that
the commutator $[T_\sigma, M_a]$
is
a compact \CZ operator.
\end{proof}

\begin{ackno}\rm
\'A.B.  acknowledges the support from
 an  AMS-Simons Research Enhancement Grant for PUI Faculty.
 T.O.~was supported by the European Research Council (grant no.~864138 ``SingStochDispDyn")
and  by the EPSRC
Mathematical Sciences
Small Grant  (grant no.~EP/Y033507/1). The authors would also like to thank the editor and the anonymous referees for their comments and suggestions.

\end{ackno}

\end{document}